\documentclass[11pt]{amsart}
\usepackage{amssymb,amsmath,txfonts}
\usepackage{hyperref}
\newtheorem{theorem}{Theorem}
\newtheorem{prop}{Proposition}
\newtheorem{lemma}{Lemma}
\newtheorem{remark}{Remark}
\newtheorem{claim}{Claim}
\newtheorem*{question}{Question}
\newtheorem*{definition}{Definition}

\numberwithin{equation}{section}

\def\XXint#1#2#3{{\setbox0=\hbox{$#1{#2#3}{\int}$}
  \vcenter{\hbox{$#2#3$}}\kern-.5\wd0}}

\author{}
\address{Department of Mathematics\\Northwestern University \\ Evanston, IL, 60208}
\email{}
\author{Gang Liu, Yuan Yuan}
\thanks{The first author is supported by National Science Foundation grant DMS-1406593. The second author is supported by National Science Foundation grant DMS-1412384 and Simons Foundation grant (\#429722 Yuan Yuan).}
\address{Department of Mathematics\\Syracuse University \\ Syracuse, NY, 13244}
\email{}
\title[Diameter Rigidity for K\"ahler manifolds with positive bisectional curvature]{Diameter Rigidity for K\"ahler manifolds with positive bisectional curvature}
\date{}
\begin{document}
\maketitle
\begin{abstract}
Let $M^n$ be a compact K\"ahler manifold with bisectional curvature bounded from below by $1$. If $diam(M) = \pi / \sqrt{2}$ and $vol(M)> vol(\mathbb{C}\mathbb{P}^n)/ 2^n$, we prove that $M$ is biholomorphically isometric to $\mathbb{C}\mathbb{P}^n$ with the standard Fubini-Study metric. 

\end{abstract}

\bigskip
\bigskip

\section{\bf{Introduction}}
In Riemannian geometry, the basic rigidity theorems under Ricci curvature lower bound are volume rigidity theorem \cite{CE}, maximal diameter theorem \cite{Ch} and Cheeger-Gromoll splitting theorem \cite{CG}. The counterpart for K\"ahler manifolds, in some sense, however, remains mysterious (cf. \cite{LW} \cite{Li1}\cite{Li2}). For instance, it is not clear to the authors whether or not the maximal volume is achieved by the Fubini-Study metric for any compact K\"ahler manifold with positive Ricci lower bound. On the other hand, Mok \cite{Mok} proved some important metric rigidity theorems in K\"ahler geometry.

In this note, we are interested in the diameter rigidity in K\"ahler geometry when the bisectional curvature has a positive lower bound.
\begin{definition}\cite{LW}\cite{TY} Let $(M, g, J)$ be a K\"ahler manifold. The bisectional curvature of $g$ is bounded below by a constant $K$ if  
$$\frac{R(Z_1, \overline{Z}_1, Z_2, \overline{Z}_2)}{\|Z_1\|^2 \|Z_2\|^2 + |\langle Z_1, \overline{Z}_2 \rangle|^2} \geq K$$
for any nonzero vectors $Z_1, Z_2 \in T^{(1, 0)}M,$ denoted by BK $\geq K$.
\end{definition}

From now on, we assume that $(M, g, J)$ has holomorphic bisectional curvature bounded below by 1, i.e. $BK \geq 1$. By the solution of the Frankel conjecture by Siu-Yau \cite{SY} and Mori \cite{Mor}, $M$ is biholomorphic to the complex projective space $\mathbb{CP}^n$. Moreover, by the volume comparison theorem proved by Li-Wang (Corollary 1.9. in \cite{LW}), the diameter $d$ of $(M, g, J)$ is bounded above by $\frac{\pi}{\sqrt{2}}$. Note that we use the normalization of metric as in \cite{TY} that is essentially the same as in \cite{LW} (up to a constant).
In view of the Cheng's maximal diameter theorem in the Riemannian case, it is natural to ask the following
\begin{question}
If the diameter of $M$ is $\frac{\pi}{\sqrt{2}}$, is $M$ isometric to $\mathbb{CP}^n$?
\end{question}
\begin{remark}
Notice that we cannot replace the bisectional curvature lower bound by Ricci curvature bound. Indeed, the canonical K\"ahler-Einstein metric on $\mathbb{CP}^1\times\mathbb{CP}^1\times\cdot\cdot\times\mathbb{CP}^1$ has diameter strictly greater than $\mathbb{CP}^n$, if we normalize the metric so that the Ricci curvature are the same. 
\end{remark}
In \cite{TY}, Tam and Yu solved the question affirmatively by assuming that there exist complex submanifolds $P$ and $Q$ of dimension $k$ and $n-k-1$ so that $d(P, Q) = \frac{\pi}{\sqrt{2}}$.
In this note, we provide another partial answer to this question:
\begin{theorem}\label{main}
Let $(M^n, g, J)$ be a compact K\"ahler manifold with $BK \geq 1$. If the diameter of $(M, g)$ is $\frac{\pi}{\sqrt{2}}$, then there exists a totally geodesic, holomorphic isometric embedding $\tau$: $\mathbb{CP}^1\to(M, g, J)$, where the metric on $\mathbb{CP}^1$ is the standard round metric with factor $\frac{1}{2}$. 
As a consequence, 
$vol(M)=\frac{vol(\mathbb{CP}^n)}{d^n}$ for some integer $d\geq 1$. In particular, the volume of $M$ can only take discrete values. If $vol(M)>\frac{vol(\mathbb{CP}^n)}{2^n}$, then $M$ is biholomorphically isometric to $\mathbb{CP}^n$ with the standard Fubini-Study metric $g_{FS}$.
\end{theorem}
 \begin{remark}
 This theorem states that counterexample (if exists) to the question may not be found by small perturbation of the Fubini-Study metric.
 \end{remark}
 
Now we sketch the simple idea of the proof. First consider the Riemannian case. The key feature is the following: Given antipodal points $p_1, p_2$ on the standard sphere, for any $x$,  \begin{equation}\label{111}d(p_1, x)+d(p_2, x) = \pi.\end{equation} Then we can apply the maximum principle for Laplacian or volume comparison to obtain the rigidity for diameter under Ricci lower bound. In standard $\mathbb{CP}^n$ case, however, (\ref{111}) is violated, unless $p, q, x$ are collinear. Thus the traditional method in the Riemannian case cannot be directly extended to K\"ahler case. By a maximum principle and the Hessian comparison theorem, we manage to find a holomorphic curve with genus zero on which (\ref{111}) holds. Combining the solution to Frankel conjecture and an elementary degree argument, we complete the proof of the theorem.



\medskip
\vskip.1in
{\bf  \quad Acknowledgment}
We would like to thank Prof. Richard Bamler, L. F. Tam, Jiaping Wang, Steve Zelditch for their interest and helpful discussions.

\section{\bf{ Proof of Theorem \ref{main}}}

Let $p_1, p_2$ be two points on $M$ realizing the diameter of $(M, g)$. Let $l$ be a minimizing normal geodesic segment joint $p_1$ and $p_2$ with $l(0)=p_1$ and $l\left(\frac{\pi}{\sqrt{2}}\right)=p_2$. Fix any point $q$ on $l$ with $q = l(t_0)$ for $0 < t_0 <  \frac{\pi}{\sqrt{2}}$. Then $d(p_1, q) = t_0$ and $d(p_2, q) = \frac{\pi}{\sqrt{2}} - t_0$. Let $U$ be a small geodesic ball centered at $q$ contained in a holomorphic coordinate chart with radius $\delta < \min\left\{t_0, \frac{\pi}{\sqrt{2}} - t_0\right\}$. Moreover, we assume that $U$ does not intersect the cut locus of $p_1$ and $p_2$. 

We define  $r_1(x) = d(x, p_1)$, $r_2(x) = d(x, p_2)$ and $u(x) = r_1(x) + r_2(x) - \frac{\pi}{\sqrt{2}}\geq 0$. Then $r_1, r_2$ and $u$ are smooth functions on $U$. 
For any $x \in U$, as $x$ is not in the cut locus of $p_1$, there exists a unique minimizing geodesic $\gamma_1$ connecting $p_1$ and $x$ such that $\gamma_1(0)=p_1$ and $\gamma_1(r_1(x))=x$. Let $X_1$ be the unit tangent vector of $\gamma_1$ at $x$. Similarly, $\gamma_2$, $X_2$ can be defined. Note that $X_1(q) = - X_2(q).$

\begin{lemma}\label{angle}
Let $\theta(x)$ be the angle at $x$ between two real unit vectors $X_1, X_2$ in the real tangent space $T_{\mathbb{R}} M$. Then there exists a constant $C>0$ (depending on $U$), such that $$\pi - C u^{\frac{1}{2}}(x) \leq \theta(x) \leq \pi.$$
\end{lemma}

\begin{proof}
Since $M$ is compact, the sectional curvature has a lower bound. The lemma simply follows from the Toponogov comparison.
\end{proof}

Let $Z_1 = \frac{1}{\sqrt{2}} (X_1 -  \sqrt{-1}J X_1) \in T^{(1, 0)}_x M, Z_2 = \frac{1}{\sqrt{2}}(X_2 -  \sqrt{-1}J X_2) \in T^{(1, 0)}_x M$. Define an operator $L$ by $$L h(x) = (\nabla_{Z_1} \nabla_{\overline{Z}_1} + \nabla_{Z_2} \nabla_{\overline{Z}_2}) h(x)$$ for smooth functions $h(x)$ on $U$.

\begin{prop}\label{est}
There exists a constant $C>0$ (depending on $U$), such that $$L u(x) \leq C u(x).$$
\end{prop}

\begin{proof} 
Let $e_1 = Z_1$ and let $\{ e_2, \cdots, e_n \}$ be parallel orthogonal along $\gamma_1$ such that $\{e_1, e_2, \cdots, e_n\}$ is an unitary frame.  Write $Z_2 = \sum_{\alpha=1}^n a_\alpha(x) e_\alpha$.

\begin{claim}There exists a constant $C>0$ (depending on $U$) such that $$\left( 1 - C u(x) \right)^{\frac{1}{2}} \leq |a_1(x)| \leq 1$$ and thus $$ \sum_{\alpha \geq 2}|a_\alpha(x)|^2 \leq C u(x).$$
\end{claim}

\noindent\emph{Proof of Claim 1:} This just follows from the lemma above.
\medskip

The complex Hessian comparison theorem derived by Tam-Yu (Theorem 2.1 in \cite{TY}) asserts $$(r_1)_{\alpha \overline{\beta}} \leq \sqrt{\frac{1}{2}} \cot\left(\sqrt{\frac{1}{2}} r_1\right) g_{\alpha \overline\beta} + \sqrt{2} \left( \cot\left( \sqrt{2} r_1 \right) - \cot \left( \sqrt{\frac{1}{2}} r_1 \right)   \right) (r_1)_\alpha (r_1)_{\bar\beta}, $$ 
Then we obtain
 $$L (r_1(x) + r_2(x)) \leq (I) + (II),$$ where $$(I) = \sqrt{\frac{1}{2}}  \left( \sum_{\alpha =2}^m  |a_{\alpha}|^2 \right)  \left( \cot\left( \sqrt{\frac{1}{2}} r_1(x) \right) + \cot\left( \sqrt{\frac{1}{2}} r_2(x) \right)\right);$$ and 

\begin{equation}\notag
(II) = \frac{1}{\sqrt{2}} \left( 1+ |a_1(x)|^2 \right) \left(  \cot \left( \sqrt{2} r_1(x) \right) +  \cot \left( \sqrt{2} r_2(x) \right) \right).
\end{equation}

Recall $U$ is a small open neighborhood of $q$. If $U$ is sufficiently small, then by Claim $1$, 
$I \leq C u(x)$, $II \leq C u(x)$.
%
%
%
%
This concludes the proof of the proposition.

\end{proof}




By the straightforward calculation we can write the complex Hessian operator $L$ as the following real second order degenerate elliptic operator on $U$. 

\begin{lemma}\label{sim}
Let $\Box_d = X_1 X_1 + (JX_1) (JX_1) + X_2 X_2 + (JX_2) (JX_2), V= - \nabla_{JX_1} JX_1 - \nabla_{JX_2} JX_2.$
Then $$L = \Box_d - V ~~\text{on}~~U.$$
\end{lemma}

\begin{proof}
The lemma follows from the straightforward calculation:
\begin{equation}
\begin{split}
L &=\nabla_{X_1} \nabla_{X_1} + \nabla_{JX_1}\nabla_{JX_1}+ \nabla_{X_2} \nabla_{X_2} + \nabla_{JX_2}\nabla_{JX_2}\\
&= \left( X_1 X_1 - \nabla_{X_1} X_1\right) + ((JX_1) (JX_1) - \nabla_{JX_1} JX_1) + (X_2 X_2 - \nabla_{X_2} X_2) + ((JX_2) (JX_2)   - \nabla_{JX_2} JX_2) \\
&= \Box_d - \nabla_{JX_1} JX_1 - \nabla_{JX_2} JX_2. 
\end{split}
\end{equation}
\end{proof}

Let $h(x)=-u(x)$ on $U$. By Proposition \ref{est} and Lemma \ref{sim}, the nonpositive function $h(x)$ satisfies the degenerate elliptic partial differential inequality
$$(L - C) h(x) \geq 0,$$
where the positive constant $C$ is from Proposition  \ref{est}. Let $S_U$ be the zero set of $h(x)$ in $U$. 
By Proposition 4 in \cite{BS} (cf. Theorem 2 in \cite{Reh}), the maximum principle asserts that $x \in S_U$ whenever $x$ can be connected from $q$ by a finite sequence of integral curves along $X_1, JX_1, X_2, JX_2$. 
For such $x$ with $u(x)=-h(x)=0$, the broken geodesic $\gamma_1 \cup \gamma_2$ is a minimizing geodesic, implying $X_1 = -X_2$. 




Let $B$ be a geodesic ball centered at $p_1$ with radius $\epsilon_0 $ less than the injectivity radius of $M$ such that $B$ is contained in a coordinate chart at $p_1$. Fix a point $q_{\lambda} = \exp_{p_1} \left(\lambda l'(0)\right) \in B$ with $d(q_{\lambda}, p_1) =\lambda  \leq \epsilon_0$. Consider the integral curve $c_{\lambda}(s)$ satisfying 
\begin{equation}\label{ode}
 \frac{d c_{\lambda}(s)}{d s} = \lambda J \nabla r_1(c_{\lambda}(s)) ~~\text{ and }~~c_{\lambda} (0) = q_\lambda .
 \end{equation}
 As  $J \nabla r_1 $ is perpendicular to $\nabla r_1 $, $d(c_{\lambda}(s), p_1)= \lambda $ for all $s$. Therefore $c_{\lambda}(s) \in B$ and $X_1 = \nabla r_1$ is always defined. 
 Let $s_0 = \sup\{a|$ there exist a smooth family of minimal geodesics $\overline{l}_b (-a<b<a)$ containing $p_1, c_{\lambda}(b), p_2$$\}$. 
As $c_{\lambda}(s)$ is joint to $q_{\lambda}$ by the integral curve along $J X_1$, by applying Proposition 4 in \cite{BS}, $s_0>0$. If $s_0$ is finite, by compactness, $\overline{l}_b$ is a smooth family of minimal geodesics for $-s_0\leq b\leq s_0$. By using the same argument, we can extend $s_0$ a little bit more. This means $s_0=+\infty$.

It is clear from the above that $\overline{l}_b$ depends on $\lambda$. Now let $\lambda\to 0^+$. Then we obtain a family of minimal geodesics $\tau_s$ connecting $p_1$ and $p_2$. Moreover, we show that the unit tangent vector of $\tau_s$ at $p_1$ is $l'(0)\cos s +Jl'(0)\sin s$. The proof is simple as the K\"ahler metric $g$ is locally Euclidean. Nevertheless we include the proof here for the sake of completeness.
Consider the variation $\gamma(s, \lambda) := c_\lambda(s)$ for $\lambda$ sufficiently small , $s \in (-\infty, \infty)$ of the base curve $\gamma(0, \lambda) = l(\lambda)$. By the regularity of the ordinary differential equation (\ref{ode}), $\gamma(s, \lambda)$ is a smooth variation. Let $x=(x_1, \cdots, x_n, x_{n+1}, \cdots, x_{2n})$ be the real coordinate of $B$ with $x(p_1)=0$ such that 
\begin{itemize}
\item $J \frac{\partial }{\partial x_\alpha} = \frac{\partial }{\partial x_{\alpha+n}}$ for $1\leq \alpha \leq n$;
\item  $l' =  \frac{\partial }{\partial x_1}$;
\item $g(x)=\sum_{1 \leq i, j \leq 2n} (\delta_{ij}+o(|x|^2)) dx_i \otimes dx_j $.
\end{itemize} Then the equation (\ref{ode}) can be written in terms of local coordinates $x(s, \lambda) := x(\gamma(s, \lambda))$:

\begin{equation}\label{ode1}
\frac{\partial x(s, \lambda)}{\partial s} = \lambda J\nabla r_1(x(s, \lambda)) ~\text{~and~}~ x(0, \lambda) = (\lambda, 0, \cdots, 0). 
\end{equation}
%
Since the K\"ahler metric $g$ is locally Euclidean, $\nabla r_1(x(s, \lambda)) = \frac{1}{ \lambda}\left(\sum_{1 \leq j \leq 2n} x_j \frac{\partial }{\partial  x_{j}}  + o(\lambda)\right)$ and $J \nabla r_1(x(s, \lambda)) = \frac{1}{\lambda}\left(\sum_{1 \leq j \leq n} (x_j \frac{\partial }{\partial x_{n+j}} - x_{n+j} \frac{\partial }{\partial x_{j}}) + o(\lambda)\right)$. Therefore, the solution of the equation (\ref{ode1}) is given by $$x(s, \lambda) = \left( \lambda \cos s  , 0, \cdots, 0,  \lambda \sin s , 0, \cdots, 0    \right) + o(\lambda).$$
Hence, for any fixed $s$, $x(\tau_s(t))=(t \cos s, , 0, \cdots, 0,  t \sin s , 0, \cdots, 0 ).$
Therefore, this family of geodesics closes up with period $2\pi$.

\begin{prop}\label{prop1}
$S=\cup_{0\leq s<2\pi}\tau_s$ is an embedded holomorphic sphere in $M$. Moreover, $S$ is totally geodesic and isometric to the standard $2$-sphere up to a factor $\frac{1}{2}$. 
\end{prop}
\begin{proof}
It is clear that the length of $\tau_s$ is constant. Let $X = \frac{\partial}{\partial t}\tau_s(t), Y = \frac{\partial}{\partial s}\tau_s(t)$ for $0\leq t\leq \frac{\pi}{\sqrt{2}}$. Then $Y$ is a Jacobi field with initial condition \begin{equation}\label{-7}Y(0) = 0, Y'(0) = JX.\end{equation} By the second variation of arc length,  for any vector field $Z$ orthogonal to $X$ along $\tau_s$ and vanishing at $p_1$ and $p_2$, \begin{equation}\label{eq0}
0\leq\int_0^{\frac{\pi}{\sqrt{2}}}|\nabla_XZ|^2-R(Z, X, X, Z)dt =: I(Z).\end{equation}
If we take $Z = \sin(\sqrt{2}t)JX$, then by $BK\geq 1$, \begin{equation}\label{-9}R(X, JX, JX, X) = 2\end{equation} along $\tau_s$. Thus \begin{equation}\label{-10}
I(\sin(\sqrt{2}t)JX)=0.\end{equation}

\begin{claim}\label{cl2}
$R(JX, X, X, Z) = 0$ for any $Z$ orthogonal to $JX$ and $X$. Equivalently, $R(Z, X)X\in$ span$\{X, JX\}^\perp$ and $R(JX, X)X\in$ span$\{JX\}$.
\end{claim}
\begin{proof}
Assume the claim is not true. Say at some $x = \tau_{s_0}(t_0)$, for some tangent vector $Z\in T_xM$, \begin{equation}\label{eq1}R(JX, X, X, Z) >0, Z\perp JX, Z\perp X.\end{equation}  It is clear that we can find $Z$ satisfying (\ref{eq1}) in a neighborhood of $x$. Say for $0<t_1<t<t_2<\frac{\pi}{\sqrt{2}}$, $s= s_0$. Thus without loss of generality, we may assume that $0<t_0<\frac{\pi}{\sqrt{2}}$.
Let us consider a cut-off function $\xi$ satisfying $\xi\geq 0$ on $[0, \frac{\pi}{\sqrt{2}}]$, and $\xi$ has compact support in $(t_1, t_2)$. Moreover, $\xi = 1$ at $t_0$. For any $\lambda\geq 0$, consider the vector field $Z_\lambda(t) = \xi(t)\lambda Z+\sin({\sqrt{2}t})JX$. Let us plug $Z_\lambda$ in (\ref{eq0}). According to (\ref{eq0}) and (\ref{-10}), $I(Z_\lambda)\geq 0$ and $I(Z_0) = 0$.
Thus \begin{equation}\frac{d}{d\lambda}|_{\lambda = 0}I(Z_\lambda) \geq 0.\end{equation}
However, by direct calculation,
\begin{equation}\frac{d}{d\lambda}|_{\lambda = 0}I(Z_\lambda) =\int_0^{\frac{\pi}{\sqrt{2}}}-2\xi(t)\sin({\sqrt{2}t})R(JX, X, X, Z)dt<0.\end{equation}
This is a contradiction. 
\end{proof}

\begin{lemma}\label{lm1}
$Y = \frac{1}{\sqrt{2}}\sin(\sqrt{2}t)JX$ on $\tau_s$. Therefore $S$ is smooth at $p_2$. $S$ is an immersed holomorphic sphere in $M$.
\end{lemma}
\begin{proof}
Set $Y = Y_1+Y_2$, where $Y_1$ is parallel to $JX$ and $Y_2$ is orthogonal to $JX$ and $X$. $Y$ satisfies the Jacobi field equation \begin{equation}\nabla_X\nabla_XY = -R(Y, X)X.\end{equation}
Let us rewrite it as \begin{equation}\nabla_X\nabla_XY_1+\nabla_X\nabla_XY_2 = -R(Y_1, X)X-R(Y_2, X)X.\end{equation} Observe that $\nabla_X\nabla_XY_1\in$ span$\{JX\}$ and $\nabla_X\nabla_XY_2\in$ span$\{X, JX\}^\perp$.
With the help of claim \ref{cl2}, we find \begin{equation}\label{-11}\nabla_X\nabla_XY_1 = -R(Y_1, X)X, \end{equation} \begin{equation}\label{-12}\nabla_X\nabla_XY_2 = -R(Y_2, X)X. \end{equation}
Notice that $Y_1(0) = 0, Y'_1(0) =JX$. With the help of (\ref{-9}) and (\ref{-11}), we find $Y_1 = \frac{1}{\sqrt{2}}\sin(\sqrt{2}t)JX$.
Also note $Y_2(0) = 0$, $Y'_2(0) = 0$. Then from (\ref{-12}) and the uniqueness of ode, we find $Y_2 \equiv 0$.
 The proof of lemma \ref{lm1} is complete.

\end{proof}

Lemma \ref{lm1} indicates a holomorphic isometry from the rescaled standard sphere to $S$.
Next we prove $S$ is embedded. Suppose $\tau_{s_1}t_1 = \tau_{s_2}t_2$. As $d(p_1, \tau_s(t)) = t$, $t_1=t_2$. We may assume $0<t_1<\frac{\pi}{\sqrt{2}}$. If $X_{s_1}t_1 \neq X_{s_2}t_2$, by standard triangle inequality, we see that $\tau_{s_1}$ cannot be a minimizing geodesic connecting $p_1$ and $p_2$. Therefore, by the uniqueness of geodesic, $\tau_{s_1}$ is the same as $\tau_{s_2}$. By checking the initial tangent vector at $p_1$, we find $s_1=s_2$ modulo $2\pi$. Now we prove that $S$ is totally geodesic. It is clear that $\nabla_XX, \nabla_XY\in$ span$\{X, Y\}$, $\nabla_YX = \nabla_XY+[Y, X]\in$span$\{X, Y\}$ and $\nabla_YY = J\nabla_Y(\frac{1}{\sqrt{2}}\sin(\sqrt{2}t)X)\in$span$\{X, Y\}$. This completes the proof of proposition \ref{prop1}.

\end{proof}

According to Mori \cite{Mor} and Siu-Yau \cite{SY} solution to the Frankel conjecture, $M$ is biholomorphic to $\mathbb{C}\mathbb{P}^n$. Proposition \ref{prop1} says $S$ is an embedded holomorphic sphere. Let us assume the degree of $S$ is $d$ for some integer $d\geq 1$. Then $Vol(M)=\frac{vol({\mathbb{CP}^n})}{d^n}$. If $d = 1$, from the volume rigidity result in \cite{LW}, $M$ is isometric to $\mathbb{CP}^n$.

\begin{remark}
To prove $d=1$, one may estimate the integration of the Ricci form on $S$. However, there are some difficulties when the points are near $p_1$ or $p_2$.
\end{remark}

\end{document}